\documentclass[12pt]{amsart}
\usepackage{amsmath}
\usepackage{graphicx}
\usepackage{hyperref}
\usepackage[latin1]{inputenc}
\usepackage{mathtools}
\usepackage{amsfonts}
\usepackage{amssymb}
\usepackage[T1]{fontenc}
\usepackage{amsthm}
\usepackage{fullpage}
\usepackage{enumitem}
\usepackage{longtable}

\newtheorem{theorem}{Theorem}[section]
\newtheorem{corollary}{Corollary}[theorem]
\newtheorem{lemma}[theorem]{Lemma}

\theoremstyle{definition}
\newtheorem{definition}{Definition}[section]

\theoremstyle{remark}
\newtheorem*{remark}{Remark}

\numberwithin{equation}{section}

\title{Construction of Permutation Polynomials over Finite Fields with the help of SCR polynomials}


\author{Bidushi Sharma}
\address{Department of Mathematical Sciences, Tezpur University, Tezpur, Assam, 784028, India}
\email{bidushisarma@gmail.com}

\author{Dhiren Kumar Basnet*}
\address{Department of Mathematical Sciences, Tezpur University, Tezpur, Assam, 784028, India}
\email{dbasnet@tezu.ernet.in}

\begin{document}	
	\begin{abstract}
		In this paper we take a deeper look at the self conjugate reciprocal (SCR) polynomials, which towards the end of the paper aid the construction of new classes of permutation polynomials of simpler forms over $\mathbb{F}_{q^{2}}$. The paper focuses on the conditions required for a certain class of degree 2 and degree 3 SCR polynomials to have no roots in $\mu_{q+1}$ (the set of $(q+1)-\emph{th}$ roots of unity), which helps in the determination  of polynomials that permute $\mathbb{F}_{q^{2}}$. In the due course we also look upon some higher degree SCR polynomials which can be reduced down to a degree 2 SCR polynnomial over both odd and even ordered fields. We further look upon the SCR polynomials of type $ax^{q+1}+bx^{q}+bx+a^{q}$ taking both the cases under consideration viz. $a\in \mathbb{F}_{q}$ and $a\in\mathbb{F}_{q^{2}}\setminus\mathbb{F}_{q}$.
	\end{abstract}

	

	\maketitle
	
	\section{Introduction}
	Let $q=p^{n}$, where $p$ is a prime, $n\in\mathbb{N}$ and $\mathbb{F}_{q}$ denotes the finite field with $q$ elements. A polynomial $f \in \mathbb{F}_{q}[x]$ is called a Permutation polynomial (PP) of $\mathbb{F}_{q}$ if  $\alpha \longmapsto f(\alpha)$ is a bijection of $\mathbb{F}_{q}.$ The study of permutation polynomial of a finite field dates back to the 19$^{th}$ century with Hermite and Dickson pioneering this area. More background material on PP's can be found in \cite{lidl}. PP's have captured the interests of researchers for its wide application in the field  of Combinatorial Design \cite{comb}, Coding Theory \cite{cod} and Cryptography \cite{cry}.\\
    The PP's of form $x^{r}h(x^{q-1})$ over $\mathbb{F}_{q^{2}}$, where $h(x)$ $\in \mathbb{F}_{q^{2}}[x]$ have been addressed in several papers following the first in \cite{zie} which were studied via the associated rational function $g(x)$ $\in$ $\mathbb{F}_{q^{2}}(x)$ that would permute the set of $(q+1)$-$\emph{th}$ roots of unity. Because of its very simple form, PP's with fewer number of terms are of extreme interest. A lot of work have already been commenced in the direction of finding permutation trinomials over $\mathbb{F}_{q^{2}}$, which can be found in \cite{trex}, \cite{pre}. Quadrinomials, which form permutations of $\mathbb{F}_{q^{2}}$ have also been studied by several authors in \cite{quadri} and \cite{d}. The goal of this paper is to give a new approach that aids the construction of  PP's of $\mathbb{F}_{q^{2}}$ with fewer terms. The idea is to use the self conjugate reciprocal (SCR) polynomials for the same. In this paper we build up
    a method to find conditions such that a SCR polynomial has no roots in $\mu_{q+1}$, the set of $(q+1)$-$\emph{th}$ roots of unity. A significant step of our method is to take the composition of our SCR polynomial with a degree one rational function which induces a bijection from $\mathbb{F}_{q}\cup\{\infty\}$ to $\mu_{q+1}$. \\

	\section{Preliminaries}

In what follows here on, we shall use the notations and terminologies of \cite{quadri}. Let us also recall a few results which aid the proof of the results  stated in Section 3 of this paper. Also note that $B^{(q)}(x)$ is a polynomial where the coefficients of $B(x)$ are raised to the power $q$ and $\mu_{q+1}$ represents the set of all the $(q+1)$-$\emph{th}$ roots of unity.

\begin{lemma}\label{2.1}\cite{quadri} If $f(x)=x^{r}B(x^{q-1})$, where $B(x) \in \mathbb{F}_{q^{2}}[x]$, r is a positive integer and $q$ a prime power, then f(x) permutes $\mathbb{F}_{q^{2}}$ if and only if $\gcd(r,q-1)=1$ and $g_{o}(x)=x^{r}B(x)^{q-1}$ permutes $\mu_{q+1}$.
\end{lemma}

 \begin{lemma}\label{2.2}\cite{quadri} If $g_{o}(x)=x^{r}B(x)^{q-1}$, where r is a positive integer, q a prime power and $B(x)\in \mathbb{F}_{q^{2}}[x]$, then $g_{o}(x)$ maps $\mu_{q+1}$ into $\mu_{q+1}\cup\{0\}$ and if B(x) has no root in $\mu_{q+1}$, then $g_{0}(x)$ induces the same function on $\mu_{q+1}$ as does $g(x)=x^{s} B^{(q)}(1/x)/B(x)$ for any integer $s \equiv r(\mod(q+1))$.\\
 In particular $g_{0}(x)$  permutes $\mu_{q+1}$ if and only if B(x) has no root in $\mu_{q+1}$ and $g(x)$ permutes $\mu_{q+1}$.
 \end{lemma}
\begin{definition}\label{2.3}\cite{quadri} {\textbf{Self Conjugate Reciprocal Polynomial}} : A non-zero polynomial $C(X) \in \mathbb{F}_{q^{2}}[x]$ of degree n is self-conjugate
reciprocal polynomial (or SCR for short) if $x^{n}C^{(q)}(1/x)=\beta C(x)$ for some $\beta \in \mathbb{F}_{q^{2}}$.
\end{definition}
 Explicitly if $C(x)=\displaystyle\sum_{i=0}^{n}c_{i}x^{i}$, is a SCR polynomial, where $c_{n}\neq 0$, then for $0\leq i \leq n/2$, we have $c_{i} \in \mathbb{F}_{q^{2}}$ and $c_{n-i}=(\beta c_{i})^{q}$. It can also be easily seen that $\beta$ $\in$ $\mu_{q+1}$.

\begin{lemma}\label{2.4}\cite{quadri} If q is even and  $C(x) =\alpha x^{2} + \beta x +\alpha^{q}$
with $\alpha \in \mathbb{F}_{q^{2}}$ 
and $\beta$ $\in \mathbb{F}_{q}^{*}$, then $C(x)$ has at least one root in
$\mu_{q+1}$ if and only if $\ Tr(\frac{\alpha^{q+1}}{\beta^{2}})=1$.
\end{lemma}
\begin{lemma}\label{2.5} \cite{zie} A degree-one $\rho$ $\in \mathbb{F}_{q^{2}}(x)$ maps $\mu_{q+1}$ to $\mathbb{F}_{q}\cup\{\infty\}$ if and only if $\rho(x) = \dfrac{\delta x - \beta \delta^{q}}{x-\beta}$ for some $\beta \in \mu_{q+1}$ and $\delta \in \mathbb{F}_{q^{2}}\setminus \mathbb{F}_{q}$.
\end{lemma}
For any positive intezer $m$, the trace function from $\mathbb{F}_{2^{m}}$ to $\mathbb{F}_{2}$, denoted by $Tr_{1}^{m}$, is the map defined as
\begin{equation*}
    Tr(x)=x+x^{2}+x^{2^{2}}+\ldots+x^{2^{m-1}}.
\end{equation*}
\begin{lemma}\label{2.6} \cite{soln} If v $\in$ $\mathbb{F}_{2^{k}}$, then $x^{2}+x=v$ has a solution in $\mathbb{F}_{2^{k}}$ if and only if  $\displaystyle Tr(v)=0$.
\end{lemma}

\begin{lemma}\label{2.7} \cite{ken}\cite{ro} Let $q=2^{m}$. Then the polynomial $x^{3}+ax+b\in\mathbb{F}_{q}[x]$ with $ab\neq 0$ has no roots in $\mathbb{F}_{q}$ if and only if
$Tr(a^{3}/b^{2})=Tr(1)$ and $t_{1}$ and $t_{2}$ are not cubes in $\mathbb{F}_{q}$ $(m$ even$)$, $\mathbb{F}_{q^{2}}$ $(m$ odd$)$, where $t_{1}, t_{2}$ are roots of $x^{2}+bx+a^{3}=0.$
\end{lemma}
We plan to use the following idea throughout the paper:\\
    We know from Lemma \ref{2.5} that, $\ell(x)=\dfrac{\delta x-\beta \delta^{q}}{x-\beta}$ induces a bijection from $\mu_{q+1}$ to $\mathbb{F}_{q}\cup\{\infty\}$, where $\beta \in \mu_{q+1}$, $\delta \in$ $\mathbb{F}_{q^{2}}\setminus\mathbb{F}_{q}$.
     Taking $\beta=1$ we get $\ell^{-1}(x)=\dfrac{\delta^{q}-x}{\delta -x}$, a bijection from  $\mathbb{F}_{q} \cup\{\infty\}$ to $\mu_{q+1}$. If $C(x)$ $\in$ $\mathbb{F}_{q^{2}}[x]$ is a SCR polynomial, then $C(x)$ has no root in  $\mu_{q+1}$ if and only if the rational function $C\circ$$\ell^{-1}$ has no root in  $\mathbb{F}_{q} \cup\{\infty\}$. We then see that for a SCR polynomial $C(x)$ of degree $n$ such that $x^{n}C^{q}(\frac{1}{x})=C(x)$, the numerator of the rational function $C\circ$$\ell^{-1}(x)$ has all its coefficients from $\mathbb{F}_{q}.$ Here on our problem of finding whether $C(x)$ has a root in $\mu_{q+1}$ gets reduced to the simpler problem of examining whether the numerator of $C\circ$$\ell^{-1}(x)$ has a root in $\mathbb{F}_{q}$ (since at $\infty$, $C\circ$$\ell^{-1}(x)$ can be addressed directly).
\maketitle

\section{Main Results}

\begin{lemma}\label{3.1}Let $C(x) \in$ $\mathbb{F}_{q^{2}}[x]$ be a SCR polynomial of degree n. Let $B(x)\in$ $\mathbb{F}_{q^{2}}[x]$ be such that x$^{r}B^{(q)}(1/x)/B(x)$ permutes $\mu_{q+1}$ and $B(x)$ has no root in $\mu_{q+1}$. Then $x^{s}B(x)^{q-1}C(x)^{q-1}$ permutes $\mu_{q+1}$ if and only if $C(x)$ has no root in $\mu_{q+1}$, where $s\equiv r+n(\mod (q+1))$.
\\

\end{lemma}
\begin{proof} We have $x^{r}B(x)^{q-1}x^{n}C(x)^{q-1}=x^{r+n}(B(x)C(x))^{q-1}$ permutes $\mu_{q+1}$ if and only if $x^{r+n}(B(x)C(X))^{(q)}/B(x)C(x)$ permutes $\mu_{q+1}$ and $B(x)C(x)$ has no root in $\mu_{q+1}$. Since $C(x)$ is an SCR polynomial of degree $n$, $x^{n}C^{q}(1/x)=\beta C(x)$ hence  $x^{r+n}(B(x)C(X))^{(q)}/B(x)C(x)=(x^{r}B^{q}(1/x)\beta)/B(x)$ permutes $\mu_{q+1}$ and $B(x)C(x)$ has no root in $\mu_{q+1}.$ Due to Lemma \ref{2.2} and  our hypothesis we get $x^{r}B(x)^{q-1}x^{n}C(x)^{q-1}$ permutes $\mu_{q+1}$ if and only if $C(x)$ has no root in $\mu_{q+1}$.

\end{proof}
\begin{theorem}\label{3.2} Let $C(x) \in$ $\mathbb{F}_{q^{2}}[x]$ be a SCR polynomial of degree $n$. Let $B(x)\in$ $\mathbb{F}_{q^{2}}[x]$ be such that x$^{r}B^{(q)}(1/x)/B(x)$ permutes $\mu_{q+1}$ and B(x) has no root in $\mu_{q+1}$, then x$^{s}B(x^{q-1})C(x^{q-1})$ permutes $\mathbb{F}_{q^{2}}$ if and only if $\gcd(s,q-1)=1$ and $C(x)$ has no root in $\mu_{q+1}$, where $s\equiv r+n(\mod (q+1))$.

\end{theorem}
\begin{proof}Follows directly from Lemma \ref{2.1} and Lemma \ref{3.1} \end{proof}

Theorem \ref{3.2}  generalises Corollary 3 from \cite{z2}.\\

\begin{remark}\label{3.3}
    Throughout our paper, we shall consider $B(x) \in \mathbb{F}_{q^{2}}[x]$ such that  $x^{r}B^{(q)}(1/x)/{B(x)}$ permutes $\mu_{q+1}$ and $B(x)$ has no root in $\mu_{q+1}$. 
\end{remark}
Following lemma is an observation made by Ding and Zieve in \cite{quadri}. For the sake of convenience we state the proof of the same.

\begin{lemma}\label{3.4} \cite{quadri} Every degree one SCR polynomial has a root in $\mu_{q+1}$.
\end{lemma}
\begin{proof} Let $C(x)=(\beta a_{1})^{q}+a_{1}x$ where $\beta$ $\in$ $\mu_{q+1}$, $a_{1}$ $\in$ $\mathbb{F}_{q^{2}}$ represents any SCR polynomial of degree one. We have $C(\alpha)=0$ ( for some $\alpha$ $\in$ $\mu_{q+1}$) if and only if $(\beta a_{1})^{q}+a_{1}\alpha=0$ if and only if  $\alpha= -\beta^{q}a_{1}^{q-1}$. Now ($-\beta^{q}a_{1}^{q-1})^{q+1}=1$ for both odd and even characteristics.
\end{proof}
Combination of Lemma \ref{3.4} and Theorem \ref{3.2} gives  rise to a class of polynomials which are never going to permute $\mathbb{F}_{q^{2}}$. 
\begin{theorem}\label{3.5}Let $B(x) \in$ $\mathbb{F}_{q^{2}}[x]$ be such that $x^{r}B^{(q)}(1/x)/{B(x)}$ permutes $\mu_{q+1}$ and $B(x)$ has no root in $\mu_{q+1}$, then $x^{s}B(x^{q-1})((\beta a_{1})^{q}+a_{1}x^{q-1})$ never permutes $\mathbb{F}_{q^{2}}$.
\end{theorem}

\begin{theorem}\label{3.6}Let char$(\mathbb{F}_{q})=2$. If $B(x)$ is as stated in Remark \ref{3.3}, then $x^{s}B(x^{q-1})(ax^{2(q-1)}+bx^{q-1}+a^{q})$ permutes $\mathbb{F}_{q^{2}}$ if and only if $\gcd(s,q-1)=1$ and $Tr(\frac{a^{q+1}}{b^{2}})\neq1$, where a $\in$ $\mathbb{F}_{q^{2}}$ $\setminus\mathbb{F}_{q}$, $b \in \mathbb{F}_{q}^{*}$ and $s\equiv r+2(\mod (q+1))$.
\end{theorem}
\begin{proof}
We can write $x^{s}B(x^{q-1})(ax^{2(q-1)}+bx^{q-1}+a^{q})=x^{s}B(x^{q-1})C(x^{q-1})$, where $C(x)=ax^{2}+bx+a^{q}$. 
Then ( since $a \in$ $\mathbb{F}_{q^{2}}$ $\setminus\mathbb{F}_{q}$, $b \in \mathbb{F}_{q}^{*}$)
\begin{center}
      
$x^{2}C^{q}(\displaystyle\frac{1}{x})=x^{2}(a^{q}\displaystyle\frac{1}{x^{2}}+b^{q}\frac{1}{x}+a^{q^{2}})=ax^{2}+bx+a^{q}=C(x)$ .
\end{center}
 Hence $C(x)$ is an SCR polynomial of degree 2.
Therefore the result follows directly from Theorem \ref{3.2} and Lemma \ref{2.4}.
\end{proof}
Lemma \ref{3.8} shall give us a condition for a SCR polynomial of degree 2 over a field of odd characteristic to have no root in $\mu_{q+1}$. A version of this result shall also appear in a paper on SCR polynomials by Ding and Zieve. Theorem \ref{3.13} discusses the case for a degree 3 SCR polynomial over a field of even characteristic with no root in $\mu_{q+1}.$ Before proceeding further we look into the following important remark.

\begin{remark}\label{3.7} If $C(x) \in \mathbb{F}_{q^{2}}[x]$ is a SCR polynomial such that $x^{n}C^{(q)}(\frac{1}{x})=C(x)$ then numerator of $C\circ\ell^{-1}(x)$ $\in$ $\mathbb{F}_{q}[x]$ where, $\ell(x)=\dfrac{\delta x- \delta^{q}}{x-1}$, $\delta\in\mathbb{F}_{q^{2}}\setminus\mathbb{F}_{q}$.
    
\end{remark}
\begin{proof} Let $C(x)=a_{0}+a_{1}x+\cdots+a_{n-1}x^{n-1}+a_{n}x^{n}.$\\
 \begin{align*}x^{n}C^{q}(1/x)=&\text{ }x^{n}(a_{0}^{q}+a_{1}^{q}/x\ldots+a_{n-1}^{q}/x^{n-1}+a_{n}^{q}/x^{n})\\
=&\text{ }a_{0}^{q}x^{n}+a_{1}^{q}x^{n-1}+\ldots+a_{n-1}^{q}x+a_{n}^{q}\\
\end{align*}
Since $C(x)$ is SCR polynomial and $C(x)=x^{n}C^{q}(1/x)$ we equate the two polynomials to get the following:\\
$a_{0}=a_{n}^{q}$, $a_{1}=a_{n-1}^{q}\ldots$, $a_{n-1}=a_{1}^{q},$ $a_{n}=a_{0}^{q}$.\\ Hence for $n$ even, $a_{0}=a_{n}^{q}$, $a_{1}=a_{n-1}^{q},\ldots$, $a_{n/2}=a_{n/2}^{q}$.\\
For $n$ odd, $a_{0}=a_{n}^{q}$, $a_{1}=a_{n-1}^{q},\ldots$, $a_{(n-1)/2}=a_{(n+1)/2}^{q}$.\\
\begin{align*}C(x)=&\text{ }a_{n}^{q}+a_{n-1}^{q}x+\ldots+a_{n-1}x^{n-1}+a_{n}x^{n}.\\
C\circ\ell^{-1}(x)=&\text{ }a_{n}^{q}+a_{n-1}^{q}(\dfrac{\delta^{q}-x}{\delta-x})+\ldots+a_{n-1}(\dfrac{\delta^{q}-x}{\delta-x})^{n-1}+a_{n}(\dfrac{\delta^{q}-x}{\delta-x})^{n}.
\end{align*}
\begin{align*}
\text{Numerator of C}\circ\ell^{-1}(x)=&\text{ }a_{n}^{q}(\delta-x)^{n}+a_{n-1}^{q}(\delta^{q}-x)(\delta-x)^{n-1}+\ldots\\
+&\text{ }a_{n-1}(\delta^{q}-x)^{n-1}(\delta-x)+a_{n}(\delta^{q}-x)^{n}.
\end{align*}
Let  the numerator of $C\circ\ell^{-1}(x)=h(x)$.\\
Then \begin{align*}\text{$h$}^{q}(x)=&\text{ }a_{n}(\delta^{q}-x)^{n}+a_{n-1}(\delta-x)(\delta^{q}-x)^{n-1}+\ldots\\
 +&\text{ }a_{n-1}^{q}(\delta-x)^{n-1}(\delta^{q}-x)
+a_{n}^{q}(\delta-x)^{n}=h(x).
\end{align*}
Hence $h(x)=h^{q}(x)\Rightarrow h(x)\in\mathbb{F}_{q}[x]$.

\end{proof}
\begin{lemma}\label{3.8} Let $\mathbb{F}_{q}$ be a field of odd characteristic. Then $C(x)=a^{q}+bx+ax^{2}$ $\in$ $\mathbb{F}_{q^{2}}[x]$, $a\in$ $\mathbb{F}_{q^{2}}\setminus\mathbb{F}_{q}$, $b\in$ $\mathbb{F}_{q}^{*}$ such that $(a^{q}+b+a)(a^{q}-b+a) \neq 0$. Then $C(x)$ has no root in $\mu_{q+1}$ if and only if $b^{2}-4a^{q+1}$ is a non zero non square in $\mathbb{F}_{q}$.
\end{lemma}
\begin{proof}

    We find that $C(1)=a^{q}+b+a\neq 0$ and $C(-1)=a^{q}-b+a\neq 0$ by our hypothesis. We now need to find a condition such that $C(x)$ has no root in $\mu_{q+1}\setminus\{1,-1\}$. In Lemma \ref{2.5}, if we fix $\beta=1$, then $\ell^{-1}(x)=\dfrac{\delta^{q}-x}{\delta -x}$ is a bijection from  $\mathbb{F}_{q} \cup\{\infty\}$ to $\mu_{q+1}$. It is observed that $C(x)$ $\in$ $\mathbb{F}_{q^{2}}[x]$ has no root in $\mu_{q+1}\setminus\{1,-1\}$ if and only if the rational function $C\circ$$\ell^{-1}$ has no root in  $\mathbb{F}_{q} \cup\{\infty\}$. \\
     At $x= \infty$, $C\circ\ell^{-1}$($\infty)=C(1)\neq0$ by hypothesis.
    What now remains is to determine when the numerator of  $C\circ$$\ell^{-1}$ has no root in $\mathbb{F}_{q}$. From Remark \ref{3.7}, it is clear that the numerator of $C\circ\ell^{-1}(x)\in \mathbb{F}_{q}[x].$ Now,
     \begin{center}  $C\circ$$\ell^{-1}(x)= a^{q}+b \displaystyle \left(  \frac{\delta^{q}-x}{\delta-x}\right)+a\left(\frac{\delta^{q}-x}{\delta-x}\right)^{2}$,
     \end{center}
     \begin{align*} \text{Numerator of C}\circ\ell^{-1}(x)=& \text{ } a^{q}(\delta-x)^{2}+b(\delta^{q}-x)(\delta-x)+a(\delta^{q}-x)^{2}\\
      =& \text{ }a^{q}(\delta^{2}-2x\delta+x^{2})+b(\delta^{q+1}-\delta^{q}x-\delta x+x^{2})\\
     &+a(\delta^{2q}-2\delta^{q}x+x^{2})\\
     =& \text{ }x^{2}(a+a^{q}+b)+x(-2a^{q}\delta-b\delta^{q}-b\delta-2\delta^{q}a)\\
     &+(b\delta^{q+1}+a^{q}\delta^{2}+a\delta^{2q}).
     \end{align*}
     \\Take $A=a+a^{q}+b$, $B=-2a^{q}\delta-b\delta^{q}-b\delta-2\delta^{q}a$ and $C=b\delta^{q+1}+a^{q}\delta^{2}+a\delta^{2q}$.
     \begin{align*} \text{We calculate the value of B$^{2}-4AC$}=& \text{ }(2a^{q}\delta+b\delta^{q})^{2}+(b\delta+2\delta^{q}a)^{2}\\
     &+2(2a^{q}\delta+b\delta^{q})(b\delta+2\delta^{q}a)\\
     &-4(a+a^{q}+b)(b\delta^{q+1}+a^{q}\delta^{2}+a\delta^{2q})\\
     =& \text{ }(b^{2}-4a^{q+1})(\delta^{2q}+\delta^{2}-2\delta^{q+1})\\
     =& \text{ }(b^{2}-4a^{q+1})(\delta^{q}-\delta)^{2}
     \end{align*}

      Hence we get that numerator of $C\circ\ell^{-1}$ has no
     root in $\mathbb{F}_{q}$ if and only if ($b^{2}-4a^{q+1})$ is a non zero non square in $\mathbb{F}_{q}.$
    
\end{proof}
\begin{theorem}\label{3.9} Let $\mathbb{F}_{q}$ be a field of odd characteristic. Let $B(x)\in \mathbb{F}_{q^{2}}[x]$ be as stated in Remark \ref{3.3}. Then $x^{s}B(x^{q-1})(a^{q}+bx^{q-1}+ax^{2(q-1)})$, where a $\in$ $\mathbb{F}_{q^{2}}$ $\setminus\mathbb{F}_{q}$, $b\in \mathbb{F}_{q}^{*}$ and $(a^{q}+b+a)(a^{q}-b+a) \neq0$, permutes $\mathbb{F}_{q^{2}}$ if and only if $\gcd(s,q-1)=1$ and ($b^{2}-4a^{q+1})$ is a non zero non square in $\mathbb{F}_{q}$, where $s\equiv r+2(\mod (q+1))$.
\end{theorem}
\begin{proof} We have
    $x^{s}B(x^{q-1})(a^{q}+bx^{q-1}+ax^{2(q-1)})=x^{s}B(x^{q-1})C(x^{q-1})$, where $C(x)=a^{q}+bx+ax^{2}$ is a SCR polynomial of degree 2 and hence by Theorem \ref{3.2} and  Lemma \ref{3.8}, the result follows.
\end{proof}
\begin{theorem}\label{3.10}
    Let $\mathbb{F}_{q}$ be a field of odd characteristic. Let $B(x)\in \mathbb{F}_{q^{2}}[x]$ be as stated in Remark \ref{3.3}. Then $x^{s}B(x^{q-1})(a^{q}x^{2q(q-1)}+bx^{{q}(q-1)}+a)$ where, $a\in$ $\mathbb{F}_{q^{2}}$ $\setminus\mathbb{F}_{q}$, $b\in \mathbb{F}_{q}^{*}$ and $(a^{q}+b+a)(a^{q}-b+a) \neq0$, permutes $\mathbb{F}_{q^{2}}$ if and only if $\gcd(s,q-1)=1$ and ($b^{2}-4a^{q+1})$ is a non zero non square in $\mathbb{F}_{q}$, where $s\equiv r+2q(\mod (q+1))$.
\end{theorem}
\begin{proof}
    If $C(x)=a^{q}x^{2q}+bx^{q}+a$, then $x^{s}B(x^{q-1})(a^{q}x^{2q(q-1)}+bx^{{q}(q-1)}+a)=x^{s}B(x^{q-1})C(x^{q-1}).$ It can be easily noticed that, $C(x)$ is a SCR polynomial of degree $2q$. Let $C(x)$ have a root $x\in \mu_{q+1}$. Then
    \begin{align*}
        &a^{q}x^{2q}+bx^{q}+a=0\\
        \iff&a^{q}x^{-2}+bx^{-1}+a=0\\
        \iff&a^{q}+bx+ax^{2}=0.
    \end{align*}
    $C(x)$ has a root in $\mu_{q+1}$ if and only if $C'(x)=ax^{2}+bx+a^{q}$ has root in $\mu_{q+1}.$
    The result follows from Theorem \ref{3.2} and Lemma \ref{3.8}.
\end{proof}
\begin{theorem}\label{3.11}
    Let char$(\mathbb{F}_{q})=2$. If $B(x)$ is as stated in Remark \ref{3.3}, then $x^{s}B(x^{q-1})(a^{q}x^{2q(q-1)}+bx^{{q}(q-1)}+a)$ permutes $\mathbb{F}_{q^{2}}$ if and only if $\gcd(s,q-1)=1$ and $Tr(\frac{a^{q+1}}{b^{2}})\neq1$, where $a\in$ $\mathbb{F}_{q^{2}}\setminus\mathbb{F}_{q}$, $b \in \mathbb{F}_{q}^{*}$ and $s\equiv r+2q (\mod (q+1))$.
\end{theorem}
\begin{proof}
    The proof follows in a similar way as in Theorem \ref{3.10} and by application of Theorem \ref{3.2} and Lemma \ref{2.4}.
\end{proof}

\begin{theorem}\label{3.12}

    Let $\mathbb{F}_{q}$ be a field of characteristic $p$. Let $B(x)\in$ $\mathbb{F}_{q^{2}}[x]$ be as stated in Remark \ref{3.3}, then $x^{s}B(x^{q-1})(ax^{3(q-1)}+a^{q})$, $a\in$ $\mathbb{F}_{q^{2}}$ $\setminus\mathbb{F}_{q}$ permutes $\mathbb{F}_{q^{2}}$ if and only if $\gcd(s,q-1)=1$ and  $-a^{q-1}$ $\notin$ $\mu_{\frac{q+1}{gcd(q+1,3)}}$, where $s\equiv r+3 (\mod (q+1))$.
\end{theorem}
\begin{proof}
    We can write $ax^{3(q-1)}+a^{q}=C(x^{q-1})$, where $C(x)=ax^{3}+a^{q}$ is a SCR polynomial. Then $C(x)$ has no root  in $\mu_{q+1}$ 
    \begin{align*} &\Leftrightarrow x^{3}\neq -a^{q-1}\\ &\Leftrightarrow -a^{q-1} \notin (\mu_{q+1})^{3}\\ &\Leftrightarrow-a^{q-1} \notin \mu_{\frac{q+1}{\gcd(q+1,3)}}.
    \end{align*}
    Now the result immediately follows from Theorem \ref{3.2}.

\end{proof}

\begin{theorem}\label{3.13} Let $\mathbb{F}_{q}$ be such that $q=2^{m}$. Let $B(x) \in$ $\mathbb{F}_{q^{2}}[x]$ be as stated in Remark \ref{3.3}. Then $x^{s}B(x^{q-1})(a^{q}+bx^{q-1}+bx^{2(q-1)}+ax^{3(q-1)})$ where, $a\in \mathbb{F}_{q^{2}} \setminus\mathbb{F}_{q}$, $b \in \mathbb{F}_{q}^{*}$, $(a^{q+1}+b^{2})\neq 0$, permutes $\mathbb{F}_{q^{2}}$ if and only if $\gcd (s,q-1)=1$,  $Tr\left(\dfrac{(a^{q}+b)^{3}(a+b)^{3}}{(a+a^{q})^{2}(a^{q+1}+b^{2})^{2}}\right)=Tr(1)$ and for any $\delta\in\mathbb{F}_{q^{2}}\setminus\mathbb{F}_{q}$, $t_{1}$,$t_{2}$ are not cubes in $\mathbb{F}_{q}$ ($m$ even), $\mathbb{F}_{q^{2}}$ ($m$ odd) where $t_{1}$ and $t_{2}$ are roots of $x^{2}+\dfrac{(\delta^{q}+\delta)^{3}(a^{q+1}+b^{2})}{(a+a^{q})^{2}}x+\dfrac{(\delta+\delta^{q})^{6}(a^{q}+b)^{3}(a+b)^{3}}{(a+a^{q})^{6}}$, where  $s\equiv r+3(\mod (q+1))$.
    
\end{theorem}
\begin{proof}
    We can write $x^{s}B(x^{q-1})(a^{q}+bx^{q-1}+bx^{2(q-1)}+ax^{3(q-1)})=x^{s}B(x^{q-1})C(x^{q-1})$, where $C(x)=a^{q}+bx+bx^{2}+ax^{3}$, $a \in \mathbb{F}_{q^{2}}$ $\setminus\mathbb{F}_{q}$, $b\in \mathbb{F}_{q}$, is a SCR polynomial of degree 3. We see that $C(1)=a+a^{q}\neq 0$ as $a \notin$ $\mathbb{F}_{q}$.  We follow the same procedure as in Lemma \ref{3.8}.
     \\ At $x=\infty$ , $C\circ\ell^{-1}(\infty) =C(1)\neq 0$.  Here on we only need to determine when $C\circ\ell^{-1}(x)$ has no root in $\mathbb{F}_{q}$ as at $\infty$, $C\circ\ell^{-1}(x)$ has already been addressed. Also from Remark \ref{3.7}, we see that the numerator of $C\circ\ell^{-1}(x)$ $\in \mathbb{F}_{q}[x]$.  Now,\\
     $C\circ \left(\displaystyle\frac{\delta^{q}-x}{\delta-x}\right)=a^{q}+b\left(\displaystyle\frac{\delta^{q}-x}{\delta-x}\right)+b\left(\displaystyle\frac{\delta^{q}-x}{\delta-x}\right)^{2}+a\left(\displaystyle\frac{\delta^{q}-x}{\delta-x}\right)^{3}$
    \begin{align*}\text{Numerator of C}\circ\ell^{-1}(x)=& \text{ }a^{q}(\delta-x)^{3}\\
    &+b(\delta^{q}-x)(\delta-x)^{2}+b(\delta^{q}-x)^{2}(\delta-x)
    +a(\delta^{q}-x)^{3}\\
    =& \text{ }(a^{q}+a)x^{3}+(a^{q}\delta+b\delta^{q}+b\delta+a\delta^{q})x^{2}+(a^{q}\delta^{2}+b\delta^{2}\\
    &+b\delta^{2q}+a\delta^{2q})x+(a^{q}\delta^{3}+b\delta^{q+2}+b\delta^{2q+1}+a\delta^{3q}).
    \end{align*}
     
     We rewrite the above in the following way:
     \begin{align*}
        x^{3}+&\displaystyle\frac{a^{q}\delta+b\delta^{q}+b\delta+a\delta^{q}}{a+a^{q}}x^{2}
        \\+&\displaystyle\frac{a^{q}\delta^{2}+b\delta^{2}+b\delta^{2q}+a\delta^{2q}}{a+a^{q}}x
         \\+&\displaystyle\frac{a^{q}\delta^{3}+b\delta^{q+2}
+b\delta^{2q+1}+a\delta^{3q}}{a+a^{q}}.
\end{align*}
Now take
\begin{center}
$\sigma_{1}=\displaystyle\frac{a^{q}\delta+b\delta^{q}+b\delta+a\delta^{q}}{a+a^{q}}$ ,
$\sigma_{2}=\displaystyle\frac{a^{q}\delta^{2}+b\delta^{2}+b\delta^{2q}+a\delta^{2q}}{a+a^{q}}$,\\
$\sigma_{3}=\displaystyle\frac{a^{q}\delta^{3}+b\delta^{q+2}+b\delta^{2q+1}+a\delta^{3q}}{a+a^{q}}$.
\end{center}
Therefore numerator of $C\circ\left(\displaystyle\frac{\delta^{q}-x}{\delta-x}\right)=0$ can be written as $x^{3}+\sigma_{1}x^{2}+\sigma_{2}x+\sigma_{3}=0$. Again replacing $x$ by $y+\sigma_{1}$ gives
\begin{equation}
y^{3}+(\sigma_{2}+\sigma_{1}^{2})y+(\sigma_{3}+\sigma_{1}\sigma_{2})=0.
\end{equation}
Calculating the values of $\sigma_{2}+\sigma_{1}^{2}$ and $\sigma_{3}+\sigma_{1}\sigma_{2}$   

\begin{align}
\sigma_{2}+\sigma_{1}^{2} =& \text{ }\displaystyle \frac{(\delta^{2q}+\delta^{2})(a^{q+1}+a^{q}b+ab+b^{2})}{(a^{q}+a)^{2}}\\
=&\text{ }\dfrac{(\delta^{q}+\delta)^{2}(a^{q}+b)(a+b)}{(a+a^{q})^{2}},\\
\sigma_{3}+\sigma_{1}\sigma_{2} =& \text{ } \displaystyle\frac{(a^{q+1}+b^{2})(\delta^{2q+1}+\delta^{q+2}+\delta^{3q}+\delta^{3})}{(a^{q}+a)^{2}}\\
=&\text{ }\dfrac{(a^{q+1}+b^{2})(\delta+\delta^{q})^{3}}{(a+a^{q})^{2}}.
\end{align}
Also notice that $(\sigma_{1}^{2}+\sigma_{2})(\sigma_{1}\sigma_{2}+\sigma_{3})\neq 0$, as otherwise $(a^{q+1}+b^{2})(a+b)(a^{q}+b)(\delta^{q}+\delta)^{5}=0$
$\Rightarrow(a^{q+1}+b^{2})=0$($\delta^{q}\neq\delta,$ $a\neq b$, $a^{q}\neq b)$, which is a contradiction to our hypothesis.\\
Lemma \ref{2.7} implies that Equation (3.1) will have no root in $\mathbb{F}_{q}$ if and only if Tr$\left(\dfrac{(\sigma_{1}^{2}+\sigma_{2})^{3}}{(\sigma_{1}\sigma_{2}+\sigma_{3})^{2}}\right)=$Tr$(1)$ and $t_{1}$, $t_{2}$ are not cubes in $\mathbb{F}_{q}$ ($m$ even), in $\mathbb{F}_{q^{2}}$ ($m$ odd), where $t_{1}, t_{2}$ are roots of $x^{2}+(\sigma_{3}+\sigma_{1}\sigma_{2})x+(\sigma_{2}+\sigma_{1}^{2})^{3}.$ From Equation (3.3) and (3.5) we get
\begin{align*}
    \text{Tr}\left(\dfrac{(\sigma_{1}^{2}+\sigma_{2})^{3}}{(\sigma_{1}\sigma_{2}+\sigma_{3})^{2}}\right)=\text{Tr}\left(\dfrac{(a^{q}+a)^{3}(a+b)^{3}}{(a^{q}+a)^{2}(a^{q+1}+b^{2})^{2}}\right)
\end{align*}
and

 \begin{align*} &x^{2}+(\sigma_{3}+\sigma_{1}\sigma_{2})x+(\sigma_{2}+\sigma_{1}^{2})^{3}\\=\text{ }&x^{2}+\dfrac{(\delta^{q}+\delta)^{3}(a^{q+1}+b^{2})}{(a+a^{q})^{2}}x+\dfrac{(\delta+\delta^{q})^{6}(a^{q}+b)^{3}(a+b)^{3}}{(a+a^{q})^{6}} \end{align*}

Hence from  Theorem \ref{3.2}, the result follows.
\end{proof}
\begin{theorem}\label{3.14} 

Let char$(\mathbb{F}_{q})\neq2$. If $B(x)$ is as stated in Remark \ref{3.3}, then $x^{s}B(x^{q-1})(2a+bx^{q(q-1)}+bx^{q-1})$, $a,b\in \mathbb{F}_{q}^{*}$, $(a+b)(a-b)\neq 0$, permutes $\mathbb{F}_{q^{2}}$ if and only if $a^{2}-b^{2}$ is a non zero non square in $\mathbb{F}_{q}$ and $\gcd(s,q-1)=1$ , where $s\equiv r+q+1 (\mod (q+1))$. 

\end{theorem} 
\begin{proof}
    Take $C(x)=ax^{q+1}+bx^{q}+bx+a^{q}$, then $C(x)$ is a SCR polynomial of degree $q+1$, then  $x^{r+q+1}B(x^{q-1})C(x^{q-1})=x^{r+q+1}B(x^{q-1})(2a+bx^{q(q-1)}+bx^{q-1})$.\\
    For any $x\in\mu_{q+1}$, $C(x)\neq 0\iff a+\dfrac{b}{x}+bx+a\neq0$ $\iff$ $C'(x)=bx^{2}+2ax+b\neq0$. Notice that $C'(x)$ is also a SCR polynomial of degree 2 and will have no root in $\mu_{q+1}$ if and only if $C(x)$ has no root in $\mu_{q+1}$. Again, due to hypothesis $C(1)=C'(1)\neq0$, $C(-1)=C'(-1)\neq0$. Now, by the procedure followed in Lemma \ref{3.8}, $C'(x) \in \mathbb{F}_{q^{2}}[x]$ has no root in $\mu_{q+1}\setminus\{1,-1\}$ if and only if the rational function $C'\circ\ell^{-1}(x)$ has no root in  $\mathbb{F}_{q} \cup\{\infty\}$, where $\ell^{-1}(x)=\dfrac{\delta^{q}-x}{\delta -x}$.\\
    At $\infty$, $C'\circ\ell^{-1}(\infty)=C'(1)=a+b\neq 0$ by hypothesis.
    \begin{equation*}
    C'\circ\ell^{-1}(x)= b(\frac{\delta^{q}-x}{\delta-x})^{2}+2a(\frac{\delta^{q}-x}{\delta-x})+b
    \end{equation*}
    \begin{align*}\text{The numerator of $C'$}\circ\ell^{-1}(x) = & \text{ } x^{2}(2b+2a)+x(-2b\delta^{q}- 2a\delta^{q}-2a\delta-2b\delta)\\
    &+ b\delta^{2q}+2a\delta^{q+1}+b\delta^{2}.
    \end{align*}
    Due to Remark \ref{3.7}, the numerator of $C'\circ\ell^{-1}(x)\in\mathbb{F}_{q}[x]$.
Now take $A=2b+2a$, $B=-(2b\delta^{q}+2a\delta^{q}+2a\delta+2b\delta)$, $C=b\delta^{2q}+2a\delta^{q+1}+b\delta^{2}$.\\
Calculating the value of $B^{2}-4AC$ we get,
\begin{align*}
    & (2b\delta^{q}+2a\delta^{q}+2a\delta+2b\delta)^{2}-4(2b+2a)(b\delta^{2q}+2a\delta^{q+1}+b\delta^{2})\\
    =& \text{ } 4\delta^{2q}(a^{2}-b^{2})+4\delta^{2}(a^{2}-b^{2})-8\delta^{q+1}(a^{2}-b^{2})\\
    =& \text{ } 4(a^{2}-b^{2})(\delta^{2}-2\delta^{q+1}+\delta^{2q})\\
    =& \text{ } 4(a^{2}-b^{2})(\delta^{q}-\delta)^{2}
\end{align*}    
Hence we get that numerator of $C'\circ\ell^{-1}(x)$ has no root in $\mathbb{F}_{q}$ if and only if $(a^{2}-b^{2})$ is a non zero non square in $\mathbb{F}_{q}.$ Hence, the result follows from Theorem \ref{3.2}.
    
    \end{proof}
    \begin{theorem}\label{3.15}
         Let $B(x)$ be as stated in Remark \ref{3.3}.\\\\
         (1) If char($\mathbb{F}_{q})\neq 2$, then $x^{s}B(x^{q-1})(a+a^{q}+bx^{q(q-1)}+bx^{q-1})$, where $a\in \mathbb{F}_{q^{2}}\setminus \mathbb{F}_{q},$ $b\in \mathbb{F}_{q}^{*}$ and $(a+a^{q}+2b)(a+a^{q}-2b)\neq 0$, permutes $\mathbb{F}_{q^{2}}$ if and only if $(a+a^{q})^{2}-(2b)^{2}$ is a non zero non square in $\mathbb{F}_{q}$ and $\gcd(s,q-1)=1$, where $s\equiv r+q+1 (\mod(q+1))$.\\\\
         (2) If char($\mathbb{F}_{q})=2$, then $x^{s}B(x^{q-1})(a+a^{q}+bx^{q(q-1)}+bx^{q-1})$, where $a\in \mathbb{F}_{q^{2}}\setminus \mathbb{F}_{q},$ $b\in \mathbb{F}_{q}^{*}$, permutes $\mathbb{F}_{q^{2}}$ if and only if Tr($\frac{b^{2}}{(a+a^{q})^{2}})\neq1$ and $\gcd(s,q-1)=1$, where $s\equiv r+q+1(\mod(q+1))$. 
        
\end{theorem}
\begin{proof}
    (1) and (2). Assume $C(x)=ax^{q+1}+bx^{q}+bx+a^{q}$, which is a SCR polynomial of degree $q+1$. Then $x^{r+q+1}B(x^{q-1})C(x^{q-1})=x^{r+q+1}B(x^{q-1})(a+a^{q}+bx^{q(q-1)}+bx^{q-1})$ and $ax^{q+1}+bx^{q}+bx+a^{q}$ has a root in $\mu_{q+1}$ if and only if $bx^{2}+(a+a^{q})x+b$ has a root in $\mu_{q+1}.$ Also, notice that $bx^{2}+(a+a^{q})x+b$ is a SCR polynomial of degree 2. Following the steps of Theorem \ref{3.14} and by application of Theorem \ref{3.2}, Theorem \ref{3.6}, Lemma \ref{3.8} we get the desired results.

\end{proof}
\begin{theorem}\label{3.16}
    Let $B(x)$ be as in Remark \ref{3.3} and $\gcd(n,q-1)=1$. Then for $s\equiv r+n(\mod(q+1))$, $\beta \in \mu_{q+1}$, $s\in \mathbb{F}_{q^{2}}\setminus\mathbb{F}_{q}$, $x^{s}B(x^{q-1})((\delta x^{q-1}-\beta\delta^{q})^{n}-(x^{q-1}-\beta)^{n})$ never permutes $\mathbb{F}_{q^{2}}$.
\end{theorem}
\begin{proof}
    Take $C(x)=(\delta x-\beta\delta^{q})^{n}-(x-\beta)^{n}$. Notice that $x^{n}C^{(q)}(\frac{1}{x})=(-1)^{n}\beta^{-n}C(x)$ and $(-1)^{n}\beta^{-n} \in \mu_{q+1}$ for both odd and even characteristic field, impliying $C(x)$ is a SCR polynomial of degree $n$. Due to Theorem \ref{3.2}, we aim to check for roots of $C(x)$ in $\mu_{q+1}$.
    \begin{align*}
        \text{$C(\beta)$}=&\text{ }(\delta\beta-\beta\delta^{q})^{n}-(\beta-\beta)^{n}\\
        =&\text{ }\beta^{n}(\delta-\delta^{q})^{n}\\
        \neq&\text{ }0\text{ } (\text{as }\delta\in\mathbb{F}_{q^{2}}\setminus\mathbb{F}_{q}).
    \end{align*}
    Let $\alpha\in\mu_{q+1}$ such that $\alpha\neq\beta$ and $C(\alpha)=0.$
    \begin{align*}
        \text{$C(\alpha)$}=0\Rightarrow(\delta\alpha-\beta\delta^{q})^{n}-(\alpha-\beta)^{n}=&\text{ }0\\
        \Rightarrow (\delta\alpha-\beta\delta^{q})^{n}=&\text{ }(\alpha-\beta)^{n}\\
        \Rightarrow \left(\dfrac{\delta\alpha-\beta\delta^{q}}{\alpha-\beta}\right)^{n}=&\text{ }1.
    \end{align*}
    Let $\zeta=\dfrac{\delta\alpha-\beta\delta^{q}}{\alpha-\beta}$, then $\zeta^{n}=1$. Again $\zeta^{q}=\zeta$, therefore we get $o(\zeta)|\gcd(n,q-1)\Rightarrow\zeta=1\Rightarrow \alpha=\dfrac{\beta(\delta^{q}-1)}{\delta-1}\in\mu_{q+1}.$ $C(x)$ always has a root in $\mu_{q+1}$ and hence the result follows from Theorem \ref{3.2}.

\end{proof}
\begin{theorem}\label{3.17}
    Let $B(x)$ be as stated in Remark \ref{3.3} and let $f(x)=x^{m}+b_{1}x^{m-1}+\cdots+b_{m-1}x+b_{m} \in \mathbb{F}_{q}[x]$ be irreducible. Then for $\beta,\delta \in \mathbb{F}_{q^{2}}$ such that $\beta^{q+1}=1$ and $\delta\notin \mathbb{F}_{q}$, $F(x)=x^{s}B(x^{q-1})((\delta x^{q-1}-\beta\delta^{q})^{m}+b_{1}(\delta x^{q-1}-\beta\delta^{q})^{m-1}(x^{q-1}-\beta)+\cdots+b_{m-1}(\delta x^{q-1}-\beta\delta^{q})(x^{q-1}-\beta)^{m-1}+b_{m}(x^{q-1}-\beta)^{m}$ permutes $\mathbb{F}_{q^{2}}$ if and only if $\gcd(s,q-1)=1$, where $s\equiv m+r(\mod(q+1))$.
\end{theorem}
\begin{proof}
    Take $C(x)=(\delta x-\beta \delta^{q})^{m}+b_{1}(\delta x-\beta \delta^{q})^{m-1}(x-\beta)+\cdots +b_{m-1}(\delta x-\beta \delta^{q})(x-\beta)^{m-1}+b_{m}(x-\beta)^{m}.$ Then $x^{m}C^{(q)}(\frac{1}{x})=(-1)^{m}\beta^{m}C(x)$ and $(-1)^{m}\beta^{-m}\in\mu_{q+1}$. Hence $C(x)$ is a SCR polynomial of degree $m$. As a consequence of Theorem \ref{3.2}, we only need to check when $C(x)$
    has a root in $\mu_{q+1}$.
    \begin{align*}
        \text{$C(\beta)$}=&\text{ }(\delta\beta-\beta\delta^{q})^{n}\\
        =&\text{ }\beta^{n}(\delta-\delta^{q})^{n}\\
        \neq&\text{ }0\text{ } (\text{as }\delta\in\mathbb{F}_{q^{2}}\setminus\mathbb{F}_{q}).
    \end{align*}
     Let $\alpha\in\mu_{q+1}$ such that $\alpha\neq\beta$ and $C(\alpha)=0\Rightarrow$
     \begin{align*}
     &(\alpha-\beta)^{m}[(\frac{\delta \alpha -\beta \delta^{q}}{\alpha -\beta})^{m}+b_{1}(\frac{\delta \alpha -\beta\delta^{q}}{\alpha-\beta})^{m-1}+...+b_{m-1}(\frac{\delta\alpha-\beta\delta^{q}}{\alpha- \beta})+b_{m}]=0\\
     &\Rightarrow f(\xi)=0\text{ where, } \xi=\frac{\delta\alpha-\beta\delta^{q}}{\alpha-\beta}.
\end{align*}
    
  Since $f$ is irreducible over $\mathbb{F}_{q}\Rightarrow$ $\xi \notin \mathbb{F}_{q}.$ This leads to a contradiction since $\xi^{q}=\xi$, implying that $C(x)$ has no root in $\mu_{q+1}$ leading to our result.
 \end{proof}

 \section{Certain Permutation polynomials over  with simple forms}
 We now use the idea that for any $r\in\mathbb{N}$ and $B(x)$ $\in$ $\mathbb{F}_{q^{2}}[x]$  such that $x^{r}B^{(q)}(1/x)/$
 $B(x)$ permutes $\mu_{q+1}$, there are infinitely many ways to choose a SCR polynomial $C(x) \in$ $\mathbb{F}_{q^{2}}[x]$, which has no root in $\mu_{q+1}$. The challenge is to select $C(x)$ such that $B(x)C(x)$ has fewer terms and use it to find several classes of permutation polynomials of $\mathbb{F}_{q^{2}}$. Corollaries \ref{4.1}-\ref{4.4} demonstrate several classes of binomials, trinomials, quadrinomials and pentanomials that permute $\mathbb{F}_{q^{2}}.$
 \begin{corollary}\label{4.1}

    Let $m\in$ $\mathbb{N}$ be arbitrary such that $\gcd(m(q-1)+1,q+1)=1$. Then we get the following classes of trinomials and quadrinomials permuting $\mathbb{F}_{q^{2}}$.
    \\
    \\ (1) Let $q=2^{k}$, then $ax^{(2+m)(q-1)+3}+bx^{(m+1)(q-1)+3}+a^{q}x^{m(q-1)+3} $,
    permutes $\mathbb{F}_{q^{2}}$ if and only if k is odd  and $Tr(\frac{a^{q+1}}{b^{2}})\neq1$, where $a\in$ $\mathbb{F}_{q^{2}}$ $\setminus\mathbb{F}_{q}$, $b \in \mathbb{F}_{q}^{*}$ .
    \\
    \\ (2) Let $\mathbb{F}_{q}$ be a field of odd characteristic. Then $ax^{(2+m)(q-1)+3}+bx^{(m+1)(q-1)+3}+a^{q}x^{m(q-1)+3}$ such that $a\in$ $\mathbb{F}_{q^{2}} \setminus\mathbb{F}_{q}$, $b\in \mathbb{F}_{q}^{*}$ and $(a^{q}+b+a)(a^{q}-b+a) \neq0$, permutes $\mathbb{F}_{q^{2}}$ if and only if (b$^{2}-4a^{q+1})$ is a non zero non square in $\mathbb{F}_{q}.$
    \\
    \\ (3) Let $q=2^{k}$. Then $ax^{(m+3)(q-1)+4}+bx^{(m+2)(q-1)+4}+bx^{(m+1)(q-1)+4}+a^{q}x^{m(q-1)+4}$ such that $a\in\mathbb{F}_{q^{2}}\setminus\mathbb{F}_{q}$, $b\in\mathbb{F}_{q}^{*}$, $a^{q+1}+b^{2}\neq0$ permutes $\mathbb{F}_{q^{2}}$ if and only if $Tr\left(\dfrac{(a^{q}+a)^{3}(a+b)^{3}}{(a^{q}+a)^{2}(a^{q+1}+b^{2})^{q}}\right)=Tr(1)$ and for any $\delta\in\mathbb{F}_{q^{2}}\setminus\mathbb{F}_{q}$, $t_{1},t_{2}$ are not cubes in $\mathbb{F}_{q}$($m$ even), $\mathbb{F}_{q^{2}}$($m$ odd) where $t_{1},t_{2}$ are roots of $x^{2}+\dfrac{(\delta^{q}+\delta)^{3}(a^{q+1}+b^{2})}{(a+a^{q})^{2}}x+\dfrac{(\delta+\delta^{q})^{6}(a^{q}+b)^{3}(a+b)^{3}}{(a+a^{q})^{6}}$.
    \\(4)  Let $\mathbb{F}_{q}$ be a field of odd characteristic. Then $2ax^{m(q-1)+q+2}+bx^{(q-1)(m+q)+q+2}+bx^{(q-1)(m+1)+q+2}$, where a, b $\in \mathbb{F}_{q}^{*}$, $(a+b)(a-b)\neq0$, permutes $\mathbb{F}_{q^{2}}$ if and only if $a^{2}-b^{2}$ is a non zero non square in $\mathbb{F}_{q}.$\\
    \\(5) Let $\mathbb{F}_{q}$ be a field of odd characteristic. Then $(a^{q}+a)x^{m(q-1)+q+2}+$
    $bx^{(q-1)(m+q)+q+2}\\+bx^{(q-1)(m+1)+q+2}$, where $a\in \mathbb{F}_{q^{2}}\setminus \mathbb{F}_{q}$, $b\in \mathbb{F}_{q}^{*}$ and $(a+a^{q}+b)(a+a^{q}-b))\neq 0$, permutes $\mathbb{F}_{q^{2}}$ if and only if $(a+a^{q})^{2}-(2b)^{2}$ is a non zero non square in $\mathbb{F}_{q}.$\\
    \\(6) Let $\mathbb{F}_{q}$ is a field of even characteristic. Then, $(a^{q}+a)x^{m(q-1)+q+2}+$
    $bx^{(q-1)(m+2)+q+2}\\+bx^{(q-1)(m+1)+q+2}$ , where $a\in \mathbb{F}_{q^{2}}\setminus \mathbb{F}_{q},$ $b\in \mathbb{F}_{q}$, permutes $\mathbb{F}_{q^{2}}$ if and only if Tr($\frac{b^{2}}{(a+a^{q})^{2}})\neq1$ \\

    \end{corollary}

    \begin{proof} Take $B(x)=x^{m}$. Then $xB(x)^{q-1}=x^{m(q-1)+1}$, which permutes $\mu_{q+1}$ due to hypothesis i.e $\gcd(m(q-1)+1,q+1)$=1.
    \\
    \\
    (1) Observe, $xB(x^{q-1})x^{2}(ax^{2(q-1)}+bx^{q-1}+a^{q})=ax^{(2+m)(q-1)+3}+bx^{(m+1)(q-1)+3}+a^{q}x^{m(q-1)+3}$.\\
    Applying Theorem \ref{3.6} the result follows.
    \newline Other cases follow in the same way applying Theorems \ref{3.9}, \ref{3.13}, \ref{3.14} and \ref{3.15} (both parts).
    
    \end{proof}
     \begin{corollary}\label{4.2}
    Let $\mathbb{F}_{q}$ be a field of  even characteristic such that $-1$ is a non zero non square in $\mu_{q+1}.$ Then 
    \begin{equation*}a^{q}x^{5}+bx^{q+4}+(a+a^{q})x^{2q+3}+ax^{4q+1}+bx^{3q+2},
    \end{equation*} where $a\in$ $\mathbb{F}_{q^{2}}$ $\setminus\mathbb{F}_{q}$, $b\in\mathbb{F}_{q}^{*}$, $(a^{q}+b+a)(a^{q}-b+a)\neq0$ permutes $\mathbb{F}_{q^{2}}$ if and only if $b^{2}-4a^{q+1}$ is a non zero non square in $\mathbb{F}_{q}.$
\end{corollary}
\begin{proof}
    Let $B(x)=x^{q+1}+x^{2}$. Then for $x\in\mu_{q+1}$, $B(x)=0$ $\iff1+x^{2}=0\iff x^{2}=-1$. Due to our hypothesis $B(x)$ has no zero in $\mu_{q+1}$. For $x\in \mu_{q+1}$,
    \begin{equation*}
        x^{3}\dfrac{(B(x))^{q}}{B(x)}=x^{3}\dfrac{(1+x^{2})^{q}}{1+x^{2}}=x
    \end{equation*}
    which permutes $\mu_{q+1}.$
    \\Take $C(x)=ax^{2}+bx+a^{q}$. Then $x^{3}B(x^{q-1})x^{2}C(x^{q-1})=(a+a^{q})x^{2q+3}+bx^{q+4}+bx^{3q+2}+ax^{4q+1}+a^{q}x^{5}$. Hence the result follows from Theorem \ref{3.9}.
\end{proof}
\begin{corollary}\label{4.3} Let $a'\in \mathbb{F}_{q^{2}}$ such that $a'+ 1\neq 0$. \\
\\(1) Let $\mathbb{F}_{q}$ be a field of even characteristic. Then $(a'+1)(ax^{3q+1}+a^{q}x^{4})$, $a\in\mathbb{F}_{q^{2}}\setminus\mathbb{F}_{q}$,  permutes $\mathbb{F}_{q^{2}}$ if and only if $-a^{q-1}$ $\notin$ $\mu_{\frac{q+1}{\gcd(q+1,3)}}$.\\
\\(2) Let $\mathbb{F}_{q}$ be a field of even characteristic. Then $(a'+1)(a+a^{q})x^{q+2}+(a'+1)bx^{q^{2}+2}+(a'+1)bx^{2q+1}$, where $a\in \mathbb{F}_{q^{2}}\setminus\mathbb{F}_{q}, b\in \mathbb{F}_{q}^{*}$, permutes $\mathbb{F}_{q^{2}}$ if and only if $Tr(\frac{b^{2}}{(a+a^{q})^{2}})\neq1$.\\
\\(3) Let  $\mathbb{F}_{q}$ be a field of odd characteristic. Then $(a'+1)(a+a^{q})x^{q+2}+(a'+1)bx^{q^{2}+2}+(a'+1)bx^{2q+1}$, where $a\in \mathbb{F}_{q^{2}}\setminus \mathbb{F}_{q},$ $b\in \mathbb{F}_{q}^{*}$ and $(a+a^{q}+2b)(a+a^{q}-b)\neq 0$, permutes $\mathbb{F}_{q^{2}}$ if and only if $(a+a^{q})^{2}-(2b)^{2}$ is a non zero non square in $\mathbb{F}_{q}$. 
        
    \end{corollary}
    \begin{proof}
        Let $B(x)=a'x^{q+1}+1$. Then $xB(x)^{q}/B(x)=x(a'^{q}+1)/(a'+1)$ permutes $\mu_{q+1}$ (as $B(x)$ has no root in $\mu_{q+1}$ due to hypothesis).\\
        Take $C(x)=ax^{3}+a^{q}$, then $x^{4}B(x^{q-1})C(x^{q-1})=(a'+1)(ax^{3q+1}+a^{q}x^{4})$. Hence the result follows by Theorem \ref{3.12}.
    \\The other cases follow from Theorem \ref{3.15}, (1 and 2).
        
    \end{proof}
    
    \begin{corollary}\label{4.4}
        Let $\mathbb{F}_{q}$ be a field of odd characteristic and $-1$ is a non square in $\mu_{q+1}.$ Then 
        \begin{center}$ax^{(q+2)(q-1)+3}+(a^{q}+a)x^{(q-1)q+3}+bx^{(q-1)^{2}+3}+a^{q}x^{(q-1)(q-2)+3}+bx^{3}$,
        \end{center}
        such that $a\in$ $\mathbb{F}_{q^{2}}$ $\setminus\mathbb{F}_{q}$, $b\in \mu_{q+1}$ and $(a^{q}+b+a)(a^{q}+a-b) \neq0$, permutes $\mathbb{F}_{q^{2}}$ if and only if $(b^{2}-4a^{q+1})$ is a non zero non square in $\mathbb{F}_{q}.$
    \end{corollary}
    \begin{proof}
        Take $B(x)=x^{q}+x^{q-2}$, then 
        $xB(x^{q-1})x^{2}((a^{q}+bx^{q-1}+ax^{2(q-1)})=ax^{(q+2)(q-1)+3}+(a^{q}+a)x^{(q-1)q+3}+bx^{(q-1)^{2}+3}+a^{q}x^{(q-1)(q-2)+3}+bx^{3}$. Now $B(x)=x^{q}+x^{q-2}$ has no root in $\mu_{q+1}$ if and only if -1 is a non square in $\mu_{q+1}$ and  $xB(x)^{q}/B(x)=x^{3}$ permutes $\mu_{q+1}$ if and only if $\gcd(3,q+1)=1$, which is always true, since $q$ is odd. Applying Theorem \ref{3.9}, the result follows.
        \end{proof}
        
        \section{Conclusion}
        This paper presents a generalised version of the results in \cite{z2}. This approach opens up a new path to construct several permutation polynomials of $\mathbb{F}_{q^{2}}$ with fewer terms with the help of SCR polynomials. The paper digs deeper onto the process of reducing the problem of determining whether a SCR polynomial has a root in the set $\mu_{q+1}$ to a rational function having a root in $\mathbb{F}_{q}\cup\{\infty\}$. We discuss degree 2 and degree 3 SCR polynomials along with a few other higher degree SCR polynomials.
        \\\\

       \textbf{Acknowledgements} The authors would like to express their sincere gratitude to Michael E. Zieve for his careful remarks and comments that helped in improving the quality of the paper.

\end{document}